\documentclass[10pt]{article}

\usepackage{amsmath,amssymb,amsthm}
\usepackage{amscd,times}
\usepackage{amsthm,amsfonts,amsmath,amscd,amssymb,times,epsfig,verbatim}
\usepackage[all]{xy}
\usepackage{enumerate,array}

\newtheorem{thm}{Theorem}
\newtheorem{prop}{Proposition}
\newtheorem{lem}{Lemma}
\newtheorem{cor}{Corollary}

\theoremstyle{remark}
\newtheorem{rem}{Remark}
\theoremstyle{definition}
\newtheorem{dfn}{Definition}
\newtheorem{exm}{Example}

\newcommand{\C}{\mathbb{C}}

\newcommand{\Q}{\mathbb{ Q}}
\newcommand{\R}{\mathbb{ R}}

\newcommand{\rk}{\operatorname{rk}}

\newcommand{\Z}{\mathbb{ Z}}

\DeclareMathOperator{\Ker}{Ker}

\DeclareMathOperator{\dg}{deg}

\title{On geometric formality of rationally elliptic manifolds in dimensions $6$ and $7$ }
\author{Svjetlana Terzi\'c}

\begin{document}

\maketitle

\begin{abstract}
We discuss the question of geometric formality for rationally elliptic manifolds of dimension $6$ and $7$. We prove that  a geometrically formal six-dimensional  biquotient  with $b_{2}=3$  has the  real cohomology of a   symmetric space. We also show that  a rationally hyperbolic  six-dimensional manifold with $b_2\leq 2$ and $b_3=0$ can not be geometrically formal. As it follows from  their real homotopy classification,  the  seven-dimensional geometrically formal rationally elliptic manifolds  have  the real  cohomology of  symmetric spaces as well.
\end{abstract}

\tableofcontents

\section{Introduction}
The notion of geometric formality of a closed compact manifold $M$  is defined  by an  existence of a metric $g$ on $M$ such that the exterior product of  harmonic forms  are  again harmonic forms.  It is proved in~\cite{K} that a geometrically formal manifold of dimension $\leq 4$ has the real cohomology of a symmetric space. Afterwards this notion has been further studied and there were provided many examples of non -geometrically formal homogeneous spaces~\cite{KT}~\cite{KT1},\cite{GNO}, but also the examples of geometrically formal homogeneous spaces which are not homotopy symmetric spaces~\cite{KT1}. The notion  of geometric formality  has also been studied from the point of view  of its relation to the different  positive curvatures~\cite{B},~\cite{AZ}.  

 In this note we investigate the question of geometric formality of rationally elliptic manifolds in  small dimensions. The reason for considering rationally elliptic manifolds is that  a rationally hyperbolic manifold  has  many relations in its real cohomology algebra comparing  to the number of generators, which  very often may appear as an obstruction to geometric formality. In addition, the same estimation on the Betti numbers  that holds for the rationally elliptic manifolds~\cite{FHT} holds for  the geometrically formal manifolds as well~\cite{K}. 

In Section~\ref{five} and Section~\ref{seven} we show  that,  from the classification of the  rationally elliptic manifolds in dimensions five  and seven  it directly follows that in these dimensions any geometrically formal manifold has the real cohomology of a symmetric space. In Section~\ref{six}  we consider the  biquotients of dimension six for which $b_2=3$  and prove that any such geometrically formal biquotient has the real  cohomology algebra  of a symmetric space. We also show that a rationally hyperbolic six-dimensional manifold with $b_2\leq 2$ and $b_3=0$ can not be geometrically formal.

{\it Acknowledgment:} The author would like to thank  the referee  whose  remarks made the author significantly  clarify some places in the paper and  improve the exposition.

\section{Rationally elliptic manifolds and geometric formality}
\subsection{Notion of geometric formality}

Let $(M, g)$  be a  closed oriented Riemannian manifold and   $\Omega ^{*}(M)$  its de Rham algebra of differential forms. A differential form
$\omega \in \Omega ^{k}(M)$ is  said to be  harmonic if 
$
\Delta \omega = d\delta \omega + \delta d\omega  = (d+\delta )^2\omega =0,
$  
where $d$  is the  exterior derivative,  $\delta$ is  coderivative and  $\Delta$ is the  Laplace-de Rham operator.
To recall this in more detail,  let 
$[ , ] : \Omega_{x} ^{k}(M) \to \R$  be the  scalar product in the space of differential forms at $T_{x}M$  defined by:
\[
[\alpha _{x}, \beta _{x}] = \frac{1}{k!}\sum _{i_1, \ldots , i_k, j_1, \ldots , j_k}g^{i_1j_1}\cdots g^{i_kj_k}a_{i_1\ldots i_k}b_{j_1\ldots j_k},
\]
where $\alpha = \frac{1}{k!}\sum\limits _{i_1,\ldots , i_k}a_{i_1\ldots i_k}dx^{i_1}\wedge\cdots \wedge dx^{i_k}$ and $\beta = \frac{1}{k!}\sum\limits _{j_1, \ldots, j_k}b_{j_1\ldots j_k}dx^{j_1}\wedge \cdots \wedge dx^{j_k}$.

The  scalar product on the space $\Omega ^{k}(M)$ is defined by 
\[
\langle\alpha ,\beta \rangle=\int _{M} [\alpha _{x},\beta
_{x}]dvol_{g} \ .
\]

The  Hodge star operator
$\ast : \Omega ^{k}(M)\to \Omega ^{n-k}(M)$, $n=\dim M$,
is defined by 
\[
\alpha _{x} \wedge (\ast \beta)_{x} = [ \alpha _{x},\beta _{x}] dvol_{g_{x}}.
\]
Then for  $\alpha \in \Omega ^{k-1}$ and $\beta \in \Omega ^{k}$ it holds
$
\langle d\alpha ,\beta \rangle = (-1)^{k}\langle \alpha ,(\ast
^{-1}d\ast) \beta \rangle .
$
It implies that the operator
$\delta = (-1)^{k}\ast ^{-1}d\ast$ is adjoint to $d$ in
the space of $k$ - forms.

Denote by $\Upsilon (M,g)\subseteq \Omega ^{*}(M)$  the graded linear subspace of harmonic forms. It is well known
that any harmonic form is closed and no harmonic form is exact. In addition,  the
Hodge theorem states that any cohomology class $[\omega]\in H^{*}(M,\R )$ contains unique  harmonic representative. Thus,  there exists an  isomorphism 
between the graded vector spaces $ \Upsilon (M,g)$  and  $H^{*}(M,\R )$.

It naturally arises the question about the existence of the metric $g$ on $M$ such that    $\Upsilon (M,g)$ has an algebra structure   under  the exterior product $\wedge$.
For a such metric the algebras  $(\Upsilon (M,g) ,\wedge)$ and $(H^{*}(M,\R ),\wedge)$ are isomorphic.  
This   is defined in~\cite{K}:
\begin{dfn} 
A Riemannian metric $g$ on $M$ is said to be formal if the 
exterior product of its harmonic forms  are
harmonic forms.
\end{dfn}
\begin{dfn}
A closed Riemannian manifold $M$ is said to be geometrically formal if it admits
a formal Riemannian metric.
\end{dfn}

The following examples of geometrically formal manifolds are well known:
the  real cohomology spheres are geometrically formal since they have, up to constant, just one harmonic form;
the symmetric spaces $G/H$ are geometrically formal  for  an an invariant metric $g$. The second one follows  from the 
observations~\cite{DNF} that any $G$-invariant form on a symmetric space $G/H$ is closed and none is exact. In addition,  invariant forms  $\Omega ^{G}(G/H)$ form an algebra under  the exterior product. Since harmonic forms for an invariant metric $g$ are $G$- invariant,  it follows  that  $\Omega ^{G}(G/H)$ coincides  with  $ \Upsilon (G/H,g)$ and, thus,   $(\Upsilon (G/H,g), \wedge)$  is an algebra.

We found useful to note the following:

\begin{lem}\label{prod}
Assume that the manifold $M$ is not geometrically formal. Then the product metric $g = g_{M}\times g_{N}$ on  $M\times N$ can not be  formal for any closed manifold $N$ and any Riemannian  metrics $g_{M}$ on $M$ and $g_{N}$ on $N$.

\end{lem}
\begin{proof}
 Assume that  product metric $g$ on  $M\times N$  is a formal  metric for  some closed manifold $N$  and  some Riemannian metrics   $g_{M}$ on $M$ and $g_{N}$ on $N$. We claim that the metric $g_{M}$ is also formal. To see that let $\alpha$ be a  harmonic form on $M$ relative  to the metric $g_{M}$ and let $\ast _{M}$ be the  corresponding star operator. Then $\alpha$ is  a harmonic form on $M\times N$ relative  to the metric $g$. Namely, since $T_{x}M$ and $T_{x}N$  are orthogonal for the metric $g$  we have that $[ \beta _{(x,y)} , \alpha _{(x, y)} ]  = [ \beta _{(x,y)}^{M}, \alpha _{x}] $, where $\alpha _{(x,y)}= \alpha _{x}$ and $\beta _{(x,y)}^{M}$ is the restriction of the from $\beta _{(x, y)}$ on $T_{x}M\subset T_{(x,y)}(M\times N)$. More precisely, if $\beta _{ (x, y)} =    \frac{1}{k!}\sum\limits _{j_1, \ldots, j_k}\sum\limits_{s=0}^{k}b_{j_1\ldots j_k}(x,y)dx^{j_1}\wedge \cdots dx^{j_s} \wedge dy^{j_{s+1}}\wedge \cdots dy^{j_k}$, then $\beta _{(x,y)}^{M} = \frac{1}{k!}\sum\limits _{j_1, \ldots, j_k}b_{j_1\ldots j_k}(x,y)dx^{j_1}\wedge \cdots \wedge dx^{j_k}$.  Since $\beta _{(x,y)}\wedge (\ast _{M}\alpha) _{x} \wedge (vol_{N})_{y}  = \beta _{(x,y)}^{M}\wedge (\ast _{M}\alpha) _{x}\wedge (vol_{N})_{y}= [\beta _{(x,y)}^{M}, \alpha _{x}](vol_{M})_{x}(vol_{N})_{y} = [\beta _{(x,y)}, \alpha _{(x,y)}](vol_{M\times N})_{(x,y)} $ we obtain that on  $M\times N$ it holds   $\ast \alpha = \ast _{M}\alpha \wedge vol_{N}$ .  It further implies that  $ d(\ast \alpha) = d(\ast _{M}\alpha )\wedge vol_{N}  \pm \ast _{M}\alpha \wedge d(vol_{N}) = 0$, since obviously $d(vol_{N})=0$ and $d(\ast _{M}\alpha)=0$. Therefore if  $\alpha$ and $\beta$ are harmonic forms on $M$ then $\alpha \wedge \beta$ is   harmonic form on $M\times N$. The restriction of $\alpha \wedge \beta$ on   $M$ is the same form, so it  follows that $\alpha \wedge \beta$  is a harmonic form on $M$ and the metric $g_{M}$  is formal, what is the  contradiction.
\end{proof}

\begin{rem}\label{observ}
 Let us point out one useful observation. Assume that a manifold  $M$ is  geometrically formal
and  consider its cohomology ring $H^{*}(M, \R)$ with its  generators and relations.
Choose harmonic form in each generator for $H^{*}(M, \R)$. Then these harmonic forms satisfy the same relations as the corresponding generators in $H^{*}(M, \R)$.
In many cases the existence of such forms leads to the contradiction meaning that the  cohomology structure is often an obstruction to geometric formality.
\end{rem}

\subsubsection{Relation between  rational formality and geometric formality}
\begin{dfn}
A manifold $M$ is formal in the sense of rational homotopy theory if $\Omega ^{*} (M)$ is weakly equivalent to $H^{*}(M,\R )$:
\begin{equation}\label{form}
(\Omega ^{*} (M), d)\leftarrow (C, d)\rightarrow (H^{*}(M), d=0),
\end{equation}
where the both homomorphisms induce isomorphisms in cohomology.
\end{dfn}

The first well known examples of formal spaces are the manifolds having free cohomology algebras, then Kaehler manifolds,   
compact symmetric spaces, etc. Note that the first proof of formality of compact symmetric spaces is based on the fact we already recalled that  an  invariant metric   on a compact symmetric space is formal. Thus, in this case to prove formality one can take $(C, d) = (\Upsilon (G/H), 0)$ in~\eqref{form}, where $\Upsilon (G/H)$ is an algebra of harmonic forms for   an  invariant metric.

In addition  it is known : all homogeneous spaces $G/H$ with $\rk H=\rk G$ are formal~\cite{ON}, all closed simply connected manifolds of dimension $\leq 6$ are formal~\cite{MN},  all closed simply connected $7$-dimensional manifolds $M$  with $b_2(M)\leq 1$ are formal~\cite{FIM}.

\begin{rem}
 A geometrically formal manifold $M$ is formal: 
\[ (\Omega (M),d)\leftarrow (\Upsilon (M),d)\rightarrow (H^{*}(M), d=0).
\]
\end{rem}

The converse is not true. For example, it is proved in~\cite{KT} that the complete  flag manifolds $SU(n+1)/T^n$ are not geometrically formal, although they are formal since $\rk SU(n+1) = \rk T^n = n$. Moreover, none of the complete flag manifolds of a simple compact Lie group is geometrically formal, although they are all formal. This is proved in~\cite{KT} for the classical Lie groups and   $G_2$ and in~\cite{GNO}  for the exceptional Lie groups. For all these spaces their  cohomology ring structure is an obstruction for geometric formality. On the other hand,  in~\cite{KT1} are provided  the series of Stiefel manifolds for which it is proved to be geometrically formal and not homotopy equivalent to a symmetric space.






\subsection{ Rationally elliptic manifolds and geometric formality}
Let $X$ be  a simply connected topological space of finite type, that is $\dim H_{k}(X)< \infty$ for any $k$ . 
\begin{dfn}
$X$ is said to be  rationally elliptic if $\rk \pi _{x}(X) = \dim _{\Q} \pi _{*}(X)\otimes \Q $ is finite and it is said to be  rationally hyperbolic if  $\rk \pi _{k}(X) = \dim _{\Q} \pi _{*}(X)\otimes  \Q $ is infinite.
\end{dfn}

\begin{exm} 
The compact homogeneous spaces and  the biquotients of compact Lie groups are rationally elliptic spaces, see~\cite{FHT}.
\end{exm}

The  ranks of the homotopy groups of a rationally elliptic space  $X$,  $\dim X=n$ satisfy~\cite{FHT}:
\begin{equation}\label{ranks}
\sum _{k}2k\cdot \rk \pi _{2k}(X) \leq n,\;\;\; \sum _{k}(2k+1)\cdot \rk \pi _{2k+1}(X)\leq 2n-1 .
\end{equation}

We want to consider the  question of geometric formality, or  more precisely the weaker question of the real cohomology structure of geometrically formal manifolds,  for rationally elliptic spaces. Why to consider rationally elliptic spaces?

The first reason comes from the fact that the Betti numbers of a geometrically formal manifold $M$  satisfy~\cite{K}: 
\[
b_{i}(M)\leq b_{i}(T^{\dim M}), \; 1\leq i\leq dim M.
\]
It implies that 
\begin{equation}\label{bound}
\sum _{i=1}^{\dim M}b_{i}(M) \leq 2^{\dim M}.
 \end{equation}
On the other hand,   it is known~\cite{FHT} that   the  Betti numbers of  a rationally elliptic space $X$  satisfy  the inequality~\eqref{bound} as well.

The second  reason is that a  rationally hyperbolic space has  many relations in its real cohomology algebra comparing  to the number of generators.
Namely, 
let us recall~\cite{FHT} that a free algebra 
$(\wedge V, d)$ is said to be a  minimal model for a commutative differential graded algebra $(\mathcal{A},d_{\mathcal{A}})$
if $d(V)\subset \wedge ^{\geq 2}V$ and there exists a morphism
$
f : (\wedge V, d)\to (\mathcal{A}, d_{\mathcal{A}})
$,
which induces an isomorphism in cohomology. The minimal model  $\mu (X)$ of a simply connected topological space $X$ of a finite type is defined to be the minimal model of $A_{PL}(X)$. It is well known that $\mu (X)$ is unique up to isomorphism and  it classifies the rational homotopy type of $X$. Moreover,  the ranks of the  homotopy groups for  $X$  are given by the numbers of the generators of the corresponding degree  in the minimal model $\mu (X)$. 

For a rationally formal simply connected space $X$, the minimal model $\mu (X)$ coincides with the minimal model of  $(H^{*}(X, \Q ), d=0)$. Therefore,  the minimal model of a formal  simply connected  formal space   can be obtained from its cohomology algebra. One just starts,  see ~\cite{FHT},  with the cohomology generators of degree two and builds up the minimal model by adding the generators of higher degree to eliminate the cohomology relations, but  in the same time  keeping the  freeness of the minimal model.
Thus,  since    for  a rationally hyperbolic formal  space $X$,  $\mu (X)$  has infinite number of generators,  the number of relations in $H^{*}(X, \Q )$  is quite large comparing to  the number of generators in  $H^{*}(X, \Q)$.

Note that $\mu (X)$ and  $\mu (X)\otimes _{\Q} \R$ have the same number of generators and $\mu (X)\otimes _{\Q} \R$   is the minimal model for  $(H^{*}(X, \R), d=0)$ for  a formal $X$. It implies that the number of relations in  $H^{*}(X, \R )$ for a rationally hyperbolic formal space $X$  is quite large as well.  
Therefore, taking into account Remark~\ref{observ}, the  rationally hyperbolic formal  manifolds  are hardly to expect to admit a formal metric.

From the side of geometry,  it is conjectured by Gromov~\cite{G}  that  the estimation~\eqref{bound} holds for positively curved manifolds, while there is also conjecture by Bott~\cite{GH}  that a simply connected  manifold which admits a metric of non-negative sectional curvature is rationally elliptic.
This brought attention to the study of the connection between positive curvature and geometric formality. In that context the following results are known.  
\begin{itemize}
\item It is proved in~\cite{B} that for a simply connected compact oriented Riemannian $4$-manifold $M$  which is geometrically formal and has non-negative sectional curvature one of the following holds:
$M$ is homeomorphic to $S^4$,  $M$ is diffeomorphic to $\C P^2$ or 
 $M$ is isometric to $S^2\times S^2$ with product metric  where both factors carry metrics with positive curvature.
\item   A homogeneous geometrically formal metric  of positive curvature is either symmetric or a metric on a rationally homology sphere, see~\cite{AZ}.
\item The normal homogeneous metric on Alloff-Wallach spaces is not geometrically formal~\cite{KT1} , but it is not positively curved as well. It is  proved in ~\cite{AZ}  that no other homogeneous metric is geometrically formal as well.
\end{itemize}


\begin{rem}
 We further discuss the notion of geometric formality for  the rationally elliptic manifolds whose  dimension is $\geq 5$, because  of the more general  result of~\cite{K} which states  that  a  closed oriented geometrically formal manifold of dimension $\leq 4$  has the real cohomology algebra of a compact globally symmetric space.
\end{rem}
\subsubsection{Five-dimensional rationally elliptic manifolds}\label{five}

The following results are known:
\begin{itemize}
\item All five-dimensional simply connected rationally elliptic manifolds have the rational homotopy type of $S^5$ or $S^2\times S^3$ (~\cite{P},~\cite{T});
\item There are four diffeomorphism  types five-dimensional biquotients~\cite{Ba}:
\[
S^5,\;\;  S^2\times S^3,\;\;  X_{-1} = SU(3)/SO(3),\;\;  X_{\infty}.
\]
\end{itemize}
The manifolds $X_{-1}$ and $X_{\infty}$ are obtained by gluing two copies of non-trivial three dimensional disc  bundles over $S^2$ along the common boundary $\C P^2\# \overline{\C P^2}$. The Wu manifold $X_{-1}$ is  real cohomology sphere $S^5$, while $H^{*}(X_{\infty}, \R ) = H^{*}(S^2\times S^3, \R )$.

Thus, all geometrically formal five-dimensional simply connected rationally elliptic manifolds have the  real cohomology of a symmetric space. Among biquotients, $S^5, S^2\times S^3$ and $X_{-1}$ are geometrically formal, while for $X_{\infty}$ it is for us an  open question.

\section{Six-dimensional rationally elliptic manifolds}\label{six}

The second Betti number of a  six-dimensional rationally elliptic manifold is  by~\eqref{ranks} less than or equal $3$ .  The following results are known:

\begin{itemize}
\item All six-dimensional rationally elliptic manifold  with $b_2\leq 1$ have the real cohomology of $S^6$, $S^3\times S^3$, $S^2\times S^4$ and $\C P^3$  (~\cite{H},~\cite{TE}).
\item All six-dimensional rationally elliptic manifold with $b_2=2$ have the real homotopy type of $\C P^2\times S^2$, $SU(3)/T^2$ or $\C P ^{3}\# \C P^{3}$ (~\cite{H}).
\item All six dimensional rationally elliptic manifolds with $b_{2}=3$ have the rational homotopy groups of $S^2\times S^2\times S^2$ (~\cite{TE}),
\end{itemize}

The  first result on the  real cohomology structure of the   geometrically  formal  rationally elliptic six-manifolds for which $b_{2} \leq 2$  is as follows~\cite{TE}:
\begin{prop}
All geometrically formal six-dimensional rationally elliptic manifolds with $b_2\leq 2$ have the real cohomology of a symmetric space.
\end{prop}
\begin{cor}
The manifolds $SU(3)/T^2$ and $\C P^3\# \C P^3$ are not geometrically formal.
\end{cor}
 We discuss  here the question of geometric formality for some simply-connected six-dimensional biquotients  for which  $b_2=3$.

Let us recall some  notions and results on general six-dimensional biquotients.  The biquotient $G/\!\!/H$  is said to be reduced if $G$ is simply-connected, $H$ is connected and no simple factor of $H$ acts transitively on any simple factor of $G$. By the result of Totaro~\cite{T1}  any compact simply-connected biquotient is diffeomorphic to reduced ones.  The biquotient is said to be decomposable if it can be obtained as the total space of $G_1/\!\!/H_{1}$  bundle over $G_{2}/\!\!/H_{2}$. It is proved~\cite{DV2} that a reduced compact simply connected six-dimensional  biquotient  $G/\!\!/H$  satisfies one of the following:  
\begin{enumerate}
\item  it  is diffeomorphic to a homogeneous space or Eschenburg inhomogeneous flag manifold $SU(3)/\!\!/T^2$;
\item  it  is decomposable;
\item it is diffeomorphic to $S^{5}\times _{T^2}S^3$ or $(S^{3})^3/\!\!/T^3$.
\end{enumerate}

The only irreducible homogeneous space of dimension  $6$ which does not have the cohomology of a symmetric space is $SU(3)/T^2$ and it is not geometrically formal. The   Eschenburg inhomogeneous flag manifold  $SU(3)/\!\!/T^2$ is neither geometrically formal  as it is proved in~\cite{KT}. 

We analyze now the following decomposable biquotients:   three $\C P^2$ bundles over $S^2$ and  infinitely many $S^2$ bundles with base a $4$-dimensional biquotient -  $\C P^2$, $S^2\times S^2$, $\C P^2\# \C P^2$, $\C P^2\# \overline {\C P^2}$.  Any bundle from the infinite families of the considered bundles has the second Betti number equal $3$.

\begin{lem}
All three $\C P^2$ bundles over $S^2$ have the real  cohomology of $\C P^2\times S^2$, that is of a symmetric space.
\end{lem}
\begin{proof}
 Any $\C P^2$- bundle $E$  over $S^2$  is obtained  as the projectivisation of rank three complex vector bundle over $S^2$.  Therefore, the integral  cohomology of its total space $M$ is generated by two generators $x$ and $y$ of degree $2$ subject to the relations
\[
x^2 =0,\;\; y^3+c_1xy^2=0.
\]
If we put $y_1= y+\frac{c_1}{3}x$ then $x$ and $y_1$ generate the real  cohomology ring of $M$  and satisfy the relations $x^2=0$, $y_1^2\neq 0$, $y_1^3 =0$  and $xy_1^2 = y^2x\neq 0$. 
\end{proof}
Note that  the cohomology structure can not be obstruction for  geometric formality of any of these bundles. The trivial bundle $S^2\times \C P^2$ is geometrically formal, while   for the other two bundles we can remark  that  if some of them admits a formal metric it  admits  a symplectic structure as well.

It is proved in~\cite{KT} that any  of the infinitely many $S^2$ bundle over $\C P^2$   is geometrically formal if and only if it is a trivial bundle $S^2\times \C P^2$. Applying the same argument as it is done in~\cite{KT} for these family of bundles,  we prove the following:
\begin{thm}\label{jedan}
None of the infinitely many non-trivial $S^2$-bundles over $\C P^2\# \C P^2$
 is geometrically formal.
\end{thm}
\begin{proof}
Let $M$ be the total  space of a    $S^2$- bundle  over $\C P^2\# \C P^2$.  Then  $M$ is the  unit sphere bundle in the associated rank $3$ vector bundle and it   is obtained by the projectivisation of rank $2$ complex vector bundle $E$.
Therefore the integral   cohomology of $M$  is given by $H^{*}(M) = H^{*}(\C P^2\# \C P^2, \Z)[y]$  subject to the relation
\begin{equation}\label{relsum}
y^2 + c_{1}(E)y + c_{2}(E) =0,
\end{equation}
where $c_1(E)$ and $c_{2}(E)$ are the pull backs of the first and second Chern classes from $H^{*}( (\C P^2\# \C P^2, \Z)$.
The cohomology ring  $H^{*}(\C P^2\# \C P^2, \Z)$ has two generators $x_1, x_2$ of degree $2$ satisfying relations $x_1^2=x_2^2$, $x_1x_2=0$ and $x_1^3=0$.  The relation~\eqref{relsum} writes as
\[
y^2 + (ax_1+bx_2)y + cx_1^2 =0\; \text{for}\; a,b,c\in \Z.
\]
Let $z= y+\frac{a}{2}x_1+\frac{b}{2}x_2$, then $z^2 = y^2 + (ax_1+bx_2)y + \frac{a^2+b^2}{4}x_1^2$. It follows that
\begin{equation}\label{relsumn}
z^2 + dx_1^2 =0,\; \text{where}\; d= c-\frac{a^2+b^2}{4}.
\end{equation}
and $x_1,x_2, z$ are the cohomology generators  for the real cohomology ring $H^{*}(M, \R)$.
We obtain that $z^2x_1=z^2x_2=0$ and $z^3 = -dzx_1^2=-dzx_2^2$, what implies  that $zx_1^2$ is top degree cohomology class. 
 
Assume that $M$ is geometrically formal. Let $\omega _{1}$ and $\eta$ be the harmonic representatives for  $x_1$ and $z$ respectively. Since $\omega _1^3=0$ it follows that the kernel foliation of $\omega _1$ is at least two-dimensional. Let $v_1, v_2$ be the independent  vectors of this foliation. From~\eqref{relsumn}
 it follows $i_{v_1}(\eta ^{2}) = 2(i_{v_1}\eta )\eta =0$.

 If $d\neq 0$ then $\eta ^3 = -d\eta \omega _1^2$ is a volume form on $M$. But, $i_{v_1}(\eta ^3) = 3(i_{v_1}\eta ) \eta ^2 =0$, what is the contradiction. 

If $d=0$ then $4c= a^2+b^2$, what implies that the integers $a$ and $b$ are even.  It further  implies that  $w_{2}(V) = c_{1}(V) \; (\text{mod}\; 2)  =0$ and  $p_{1}(V) = c_{1}^{2}(E) -4c_{2}(E) = (a^2+b^2-4c)x_1^2=0$. Therefore,   by~\cite{PO},~\cite{DW} the bundle  $M$ is trivial that is   $M= S^2\times (\C P^2\# \C P^2)$ .  The connected sum   $\C P^2\# \C P^2$ is not geometrically formal, since it is known not to admit  a symplectic structure. It follows by Lemma~\ref{prod}  that no product metric  on $M$ is formal.

\end{proof}

\begin{thm}\label{dva}
None of the infinitely many $S^2$-bundles over $S^2\times S^2$ which does not have the real cohomology of $(S^2)^3$ is geometrically formal.
\end{thm}
\begin{proof} As previously, the bundle $M$  is obtained by the  projectivisation of rank $2$  complex vector bundle $E$. The integral   cohomology of $M$  is given by $H^{*}(M) = H^{*}(S^2\times S^2, \Z)[y]$   subject to the relation:
\begin{equation}\label{rels}
y^2 + (ax_1+bx_2)y + cx_1x_2 =0\; \text{for}\; a,b,c\in \Z, 
\end{equation}
where $x_1, x_2$  are the pull backs of the generators of the cohomology ring $H^{*}(S^2\times S^2, \Z)$ and they satisfy relations $x_1^2=x_2^2=0$. Let $z = y+\frac{a}{2}x_1+\frac{b}{2}x_2$. 
Then  $x_1,x_2$ and  $z$ represent  the generators for $H^{*}(M,  \R)$  and in terms of these generators the relation~\eqref{rels} writes as
\begin{equation}\label{fin}
z^2  + qx_1x_2=0,
\end{equation}
where $q= c -\frac{ab}{2}$.  Since $z^2x_1=z^2x_2=0$ we conclude that $x_1x_2z$ is non-zero top-degree cohomology class on $M$. 

Assume that $M$ is geometrically formal and let $\omega$, $\eta_1$ and $\eta _{2}$  be the harmonic representatives for $z$, $x_1$ and  $x_2$. We have that  $\eta _1 ^2=\eta _{2}^2=0$,    what implies that there exist linearly independent  vector fields  $v_1$ and $v_2$  in the  intersection of the kernel  foliations for $\eta _1$ and $\eta _2$. It   follows from~\eqref{fin} that $i_{v_1}i_{v_2}\omega ^2 =0$, so $\omega ^2\eta _1$ and $\omega ^2\eta _2$  can not be the volume forms on $M$. Thus,   the volume form must be  $\omega \eta_1 \eta_2$.

If $q\neq 0$ in~\eqref{fin} then   it is easy to see that $M$ does not  have the real cohomology of $S^2\times S^2\times S^2$. The assumption that $M$ is geometrically formal implies that $\omega  ^3 $ is a volume form on $M$ as well, what is in contradiction with the fact
$i_{v_1,v_2}\omega  ^2=0$.

If $q=0$ then $M$ has the real cohomology of $S^2\times S^2\times S^2$. In this case we have that   $ab=2c$, what implies that $p_{1}(V)=c_1^2(E)-4c_{2}(E) = (ax_1+bx_2)^2 -  4cx_1x_2 = 0$.  Note that if the both integers $a$ and $b$  are even then $w_{2}(V) = 0 $ what implies that this bundle  is trivial,  that is $M=S^2\times S^2\times S^2$, which is  geometrically formal symmetric space.

\end{proof}

\begin{thm}\label{tri}
None of the infinitely many $S^2$-bundles over   $\C P^2 \# \overline{\C P^2}$ which does not have the real  cohomology of $(S^2)^3$ 
 is geometrically formal.
\end{thm}
\begin{proof}
Let  $M$ be   the total  space of a    $S^2$- bundle over $\C P^2\# \overline{\C P^2}$. 
The real cohomology ring  for  $\C P^2\# \overline{\C P^2}$ is the same as for $S^2\times S^2$. Therefore, as in the proof of previous theorem, we conclude that if   $M$ does not have the real cohomology of $S^2\times S^2\times S^2$ then $M$ can not be geometrically formal. 

Let  $c_{1}(E) = ax_1+bx_2$ and $c_{2}(E) = cx_1^2$ are the the pullbacks of the first and the second Chern classes for $E$, where $x_1$ and $x_2$ are the generators for  $H^{*}(\C P^2\# \overline{\C P^2}, \Z)$. Then, as previously,  the real cohomology ring for $M$ is also generated by  $x_1, x_2$ and $z$ such that $z^2+dx_1^2=0$, where $d=c-\frac{a^2-b^2}{4}$. It implies that $M$ has the real cohomology of $(S^{2})^3$ if and only if $4c=a^2-b^2$. In this case $p_{1}(V) = 0$ and also the integers $a$ and $b$ are of the same parity. If the both $a$ and $b$ are even then $w_{2}(V)=0$ and the bundle $E$ is trivial, 
 that is    $M= S^2\times (\C P^2\# \overline{\C P^2})$. It is proved in~\cite{K1} that  $\C P^2\# \overline{\C P^2}$ admits no formal metric, what implies that no product metric on  $M$ is  formal.
\end{proof}

\begin{cor}\label{biquot}
None of the biquotients from the  infinite families of  the six-dimensional biquotients of the form $(SU(2))^{3}/\!\!/T^3$ different from $S^2\times S^2\times S^2$  is geometrically formal.
\end{cor}
\begin{proof}
 The six-dimensional biquotients of the form $(SU(2))^{3}/\!\!/T^3$  are classified in~\cite{DV2}. They are parametrized by  the  three families of infinite matrices and four sporadic matrices. In the same paper it is established which of these biquotients that correspond to these matrices are diffeomorphic. For the  biquotients considered in~\cite{T} which give one family of these biquotients and whose real cohomology ring has three generators $x_1,x_2, x_3$ subject to the  relations $x_1^2=0,\; ax_1x_2+x_2^2+x_2x_3=0,\; bx_1x_3+2x_2x_3+x_3^2=0$ , it is proved in~\cite{KT1} that they are not geometrically formal.  We want to remark that, as it is pointed in~\cite{DV2}, these biquotients are decomposable meaning that  any of them can be obtained as $S^2$ bundle over $S^2\times S^2$ or as  $S^2$ bundle over $\C P^2\# \overline{\C P^2}$, so Theorem~\ref{dva} and Theorem~\ref{tri} provide the new proof that they are not geometrically formal.  The other family of these biquotients has the  real cohomology ring generated by $x_1, x_2, x_3$ subject to the relations $x_1^2=0,\;x_2^2+bx_1x_2=0,\;x_3^2+c_1x_1x_3+c_2x_2x_x=0$. This family  is considered in~\cite{TE}, where it is proved that none  of these  biquotients which does not have the real cohomology of $(S^2)^{3}$ is  not geometrically formal.  The third family of the biquotients has the  real cohomology generators $x_1,x_2, x_3$ subject to the relations 
$x_1^2 + 2x_1x_2=0,\; x_2^2 +x_1x_2=0,\; x_3^2 + c_1x_1x_3 + c_2x_2x_3 =0$. For this family  it is proved in~\cite{DV2} that they are decomposable meaning that they can be represented as $S^2$ bundles over $\C P^2\# \C P^2$, so Theorem~\ref{jedan} proves that they are not geometrically formal. 
\end{proof}

\begin{rem}
Note that the biquotients from Corollary~\ref{biquot} belong to the third case in the description of six-dimensional biquotients that is given  in~\cite{DV2}.   This condition also describes the manifold $\C P^3\# \C P^3$.
\end{rem}

\subsection{On some hyperbolic six-dimensional manifolds}
We show that  none of  the  hyperbolic, closed, simply-connected six-dimensional  manifold for which   $b_{2}(M)\leq 2$ and $b_{3}(M)=0$ can be geometrically formal because of its cohomology structure. For such a  manifold it is known~\cite{H} that it is  rationally hyperbolic if and only if it has the real homotopy type of $(S^2\times S^4)\# \C P^3$ or $(S^2\times S^4)\# (S^2\times S^4)$. 

\begin{prop}
A  manifold having real cohomology structure of $(S^2\times S^4)\# \C P^3$ or $(S^2\times S^4)\# (S^2\times S^4)$  can not be geometrically formal.
\end{prop}
\begin{proof}
The  manifold $(S^2\times S^4)\# \C P^3$  has three cohomology generators $x, y, z$, such that $\deg x=\deg z = 2$ and $\deg y=4$ and $x^2=0$, $xz=0$ and $yz=0$. 
If assume  that this manifold is geometrically formal,  we have harmonic forms $\alpha, \beta $ and $\gamma$ representing the classes $x, y,z$ respectively, which satisfy the same relations as these classes.  Since $\alpha ^{2} =0$ this form have four-dimensional kernel  foliation. Denote by $v_1, v_2, v_3, v_4$ linearly  independent vectors of this foliation. 
Since $\alpha \gamma =0$ we obtain that  $0=i_{v_i}i_{v_j}(\alpha \gamma ) = \alpha \gamma (v_i, v_j)$, what implies $\gamma (v_i, v_j) =0$. It further gives 
\[
 i_{v_4}i_{v_3}i_{v_2}i_{v_1}(\gamma ^{3}) = i_{v_4}i_{v_3}i_{v_2}(3i_{v_1}(\gamma)\gamma ^2) = 3i_{v_4}i_{v_3}(\gamma (v_1,v_2)\gamma ^{2}) -2 i_{v_1}(\gamma) i_{v_2}(\gamma)\gamma)=
\]
\[ -6i_{v_4}(\gamma(v_1,v_3) - i_{v_1}(\gamma)(\gamma(v_2,v_3)\gamma  - i_{v_2}(\gamma)i_{v_3}(\gamma)) = 
-6i_{v_4}(i_{v_1}\gamma i_{v_2}\gamma i_{v_3}\gamma )=0.
\]
This is in contradiction with the fact that $\gamma ^{3}$ is a volume form.

The manifold  $(S^2\times S^4)\# (S^2\times S^4)$ has four cohomology generators $x_1, x_2,y_1,y_2$ such that $\deg x_1=\deg x_2=2$ and $\deg y_1=\deg y_2=4$, which satisfy relations $x_1^2=x_2^2=0$, $x_1x_2=0$ and  $x_1y_2=x_2y_1=0$.  
 If this manifold is geometrically formal, we would have that the  harmonic forms $\alpha _{1}$ and $\alpha _{2}$, which represent  the cohomology classes $x_1$ and $x_2$,  satisfy $\alpha _{1}^2 =\alpha _{2}^2=0$. Therefore, the kernel foliations for $\alpha _{1}$ and $\alpha _{2}$   are  four-dimensional. We  denote their basis by $v_1,v_2,v_3,v_4$ and $u_1,u_2,u_3,u_4$ respectively.

 Let $\beta _{1}$ and $\beta _{2}$ be harmonic representatives for $y_1$ and $y_2$.   Since $\alpha _{1}\beta _{2} = 0$  and $\alpha _{2}\beta _{1}=0$ we obtain that $\beta _{2}(v_1,v_2,v_3,v_4)=0$ and $\beta _{1}(u_1, u_2,u_3,u_4)=0$. 

The intersection of the kernel foliations for $\alpha _{1}$ and $\alpha _{2}$ is at least two-dimensional.  Note that this  kernel foliations can  not coincide since  it gives contradiction with the fact that $\alpha _1\beta _{1}$ and $\alpha _{2}\beta _{2}$ are volume forms.
  
Assume that the kernel intersection is two-dimensional and let   $v_1 =u_1$ and $v_2=u_2$ be the basis  of this intersection.  Since $\alpha _{1}\alpha _{2}=0$,  we obtain that 
$0=i_{v_3,v_4}(\alpha _1\alpha _2) = \alpha _{2}(v_3,v_4)\alpha _{1}$ what gives $\alpha _{2}(v_3,v_4)=0$.  Therefore, $(\alpha _{2}\beta _{2})(u_1,u_2,u_3,u_4,v_3,v_4) = \alpha _{2}(v_3,v_4)\beta _{2}(u_1,u_2,u_3,u_4) = 0$, which is in  contradiction with the fact that $\alpha _{2}\beta _{2}$ is a volume form. 

If the kernel intersection is three dimensional, let $v_1=u_1, v_2=u_2$,  $u_3=v_3$ and denote by $v_4\in \Ker (\alpha _{1}), v_{4}\notin \Ker (\alpha _{2})$ and $u_4\in \Ker (\alpha _{1}), u_{4}\in \Ker (\alpha _{2}), u_{4}\notin \Ker (\alpha _{1})$. Then from $\alpha _{1}\beta _{2} = 0$  it follows that $i_{u_4}\alpha _{1}\beta _{2}+\alpha _{1}i_{u_4}\beta _{2}= 0$ and from   $\alpha _{1}\alpha _{2} = 0$ it follows that $i_{u_4}\alpha _1 i_{v_4}\alpha _2 = 0$. Further, there  exists  vector field $x$ orthogonal to the sum of these foliations $\Ker (\alpha _{1})\oplus \Ker (\alpha _{2})$. We obtain that   $\alpha _{1}(u_4, x)i_{v_4}\alpha _{2} -\alpha _{2}(v_4, x)i_{u_4}\alpha _{1}=0$. Note that $\alpha _{1}(u_4, x), \alpha _{2}(v_4, x)\neq 0$  since, say,   for $\alpha _{2}(v_4, x)=0$ we would have 
$\alpha _{2}\beta _{2}(u_1,u_2,u_3, u_4, v_4, x) = \alpha _{2}(v_4, x)\beta _{2}(u_1,u_2,u_3,u_4)=0$,  which is  in contradiction with $\alpha _{2}\beta _{2}$ being  volume form.  Therefore, $i_{u_4}\alpha _{1} = \frac{\alpha _{1}(u_4, x)}{\alpha _{2}(v_4, x)}i_{v_4}\alpha _{2}$ what, together with previous,  implies  $\frac{\alpha _{1}(u_4, x)}{\alpha _{2}(v_4, x)}i_{v_4}\alpha _{2}\beta _{2}+\alpha _{1}i_{u_4}\beta _{2} = 0$. Therefore, we obtain that $i_{v_4}\alpha _{2}i_{u_4}\beta _{2}=0$ contradicting that $\alpha _{2}\beta _{2}$ is a volume form.

\end{proof}

\section{Seven-dimensional rationally elliptic manifolds}\label{seven}

 It is proved in~\cite{H} that a closed simply-connected seven-dimensional manifold is rationally elliptic if and only if it has the real homotopy type of one of the following manifolds : $S^7$, $S^2\times S^5$, $\C P^2 \times S^3$, $S^3\times S^4$, $N^7$, $S^3\times (\C P^2\# \C P^2) $ or $ S^3\times (\C P^2\# \overline{\C P^2})$. Here 
the manifold $N^7$ is a homogeneous space $(SU(2))^3/T^2$, where the embedding $T^2\subset (SU(2))^3$ is given by 
\[
\left\{\begin{array}{ccc}
\left(
\begin{array}{cc}
z & 0\\
0 & z^{-1}
\end{array}\right), &
\left(
\begin{array}{cc}
w & 0\\
0 & w^{-1}
\end{array}\right), &
\left(
\begin{array}{cc}
zw & 0\\
0 & (zw)^{-1}
\end{array}\right)
\end{array}\right\}.
\]

The manifolds $S^7$, $S^2\times S^5$, $\C P^2 \times S^3$ and  $S^3\times S^4$ are obviously geometrically formal. On the other side,  not all manifolds having the real homotopy types  of  these manifolds are geometrically formal. The Alloff-Wallach spaces $SU(3)/T^1$  have the real  cohomology of $S^2\times S^5$, but  the normal homogeneous metrics on these spaces  are  not formal~\cite{KT1}. This result, as we already mentioned,  is recently strengthened in~\cite{AZ}, where it is proved that   none of the homogeneous metrics on Alloff-Wallach spaces can be geometrically formal.

The real cohomology algebra for $N^7$  is as follows:
\[
\R [x_1,x_2]\otimes \wedge (y_1,y_2,y_3), dx_1=dx_2=0, dy_1= x_1^2, dy_2=x_2^2, dy_3=(x_1+x_2)^2,
\]
where $\dg x_1=\dg x_2=2$.  It follows that $N^7$ is not Cartan pair homogeneous space and, thus, not  formal in the sense of rational homotopy theory~\cite{ON}. Therefore,  it can not be geometrically formal.

The product metric on any of manifolds  $S^3\times (\C P^2\# \C P^2) $ and  $ S^3\times (\C P^2\# \overline{\C P^2})$  can not be formal  since, otherwise, it would by Lemma~\ref{prod} imply that the connected sums
$\C P^2\# \C P^2$  and $\C P^2\# \overline{\C P^2}$ are geometrically formal  manifolds which is, as we already noted,  not the case.
 
\vspace*{0.2cm}



\bibliographystyle{amsplain}

\begin{thebibliography}{23}


\bibitem{AZ}
M.~Amann and W.Ziller, \emph{Geometrically formal metrics of positive sectional curvature}, Journal of Geometric Analysis~{\bf 26} (2016), 996-1010. 

\bibitem{B}
 C.~Baer, \emph{Geometrically formal $4$-manifolds with nonnegative sectional curvature}, Communication in Analysis and Geometry~{\bf 23} (2015), no. 3, 479 - 497.

\bibitem{Ba} 
D.~Barden, \emph{Simply-connected five manifolds}. Annals of Mathematics~{\bf 82} (1965), 365 -- 385.


\bibitem{DV2}
 J.~DeVito,  \emph{The classification of compact simply connected biquotients in dimension 6 and 7}, Mathematische Annalen~{\bf 368}, no. 3-4, 1493--1541.

\bibitem{DW}
A.~Dold, H.~Whitney, \emph{ Classification of oriented sphere bundles over $4$-complex}, Annals of Mathematics (2)~{\bf 69} (1959), 667--677.


\bibitem{DNF}
 B.~A.~Dubrovin, A.~T.~Fomenko and S.~P.~Novikov, \emph{Modern
Geometry
--- Methods and Applications}, Vol.~III, Graduate Texts in
Mathematics 124, Springer Verlag 1990.

\bibitem{E}
 J.~H.~Eschenburg, \emph{Cohomology of biquotients}, Manuscripta
Mathematica~{\bf 75} (1992), 151--166.

\bibitem{FHT}
 Y.~Felix, S.~Halperin and J.~C.~Thomas,
\emph{Rational Homotopy Theory}, Springer Verlag, 2000.

\bibitem{FOT}
Y.~Felix, J.~Oprea and D.~Tanre, \emph{ Algebraic Models in Geometry}, Oxford University Press, 2008.



\bibitem{FIM}
M.~Fern\' andez, S.~Ivanov and V. Munoz, \emph{ Formality of $7$-dimensional $3$-Sasakian manifolds}, (2015), available at  arXiv:1511.08930.
 

\bibitem{GNO}
L.~Grama, C.~J.~C.~ Negreiros and A.~R.~ Oliveira, \emph{Invariant almost complex geometry on flag manifolds: geometric formality and Chern numbers},  Annali di Matematica Pura ed Applicata~{\bf 196} (2017),  no. 1, 165--200.

\bibitem{G}
M.~Gromov, \emph{Curvature, diameter and Betti numbers}, Commentarii Mathematici Helvetici~{\bf 56} (1981), 179--195.

\bibitem{GH}
K.~Grove and S.~Halperin, \emph{Contributions of rational homotopy theory to global problems in geometry},  Publ.~Math.~IHES~{\bf 56}, (1982), 171--177.


 

\bibitem{H}
M.~Herrmann, \emph{Classification and characterization of rationally elliptic manifolds in low dimensions}, (2015), available at  	arXiv:1409.8036. 

\bibitem{K}
 D.~Kotschick, \emph{ On products of harmonic forms}, Duke Mathematical Journal~{\bf 107} (2001), no. 3, 521--532.

\bibitem{K1}
 D.~Kotschick, \emph{ Geometric formality and non-negative scalar curvature}, (2013), available at  arXiv:1212.3317.

\bibitem{KT}
 D.~Kotschick and S.~Terzi\'c, \emph{On formality of generalised symmetric spaces},
Mathematical Proceedings of the Cambridge Philosophical Society~{\bf 134} (2003), 491--505.

\bibitem{KT1}
 D.~Kotschick and S.~Terzi\'c, \emph{ On geometric formality of homogeneous spaces and of biquotients},
Pacific Journal of Mathematics~{\bf 249} (2011), no. 1, 157--176. 

\bibitem{MN}
 T.~Miller, J.~Neisendorfer, \emph{ Formal and coformal spaces}, Illinois Journal of Mathematics~{\bf 22} (1978), 565--580.

\bibitem{ON}
A.~L.~Onishchik, \emph{ Topology of transitive transformation groups} (Russian), Fizmatlit Nauka Moscow 1995.

\bibitem{P}
A.~V.~Pavlov, \emph{ Five-dimensional biquotients}, Siberian Mathematical Journal~{\bf 45} (2004), no. 6, 1080--1083.

\bibitem{PO}
L.~Pontrjagin, \emph{ Classification of some skew products}, C.~R.~(Doklady) Acad.~Sci.~URSS (N.~S.) {\bf 47} (1945), 322--325.
\bibitem{TE}
S. Terzi\'c, \emph{ Geometric formality of rationally elliptic manifolds in small dimensions}, Glasnik of the Section of Natural Sciences, Montenegrin Academy of Sciences and Arts~{\bf 20} (2014), 131--145.

\bibitem{T}
 B.~Totaro, \emph{ Curvature, diameter, and quotient manifolds},
Mathematical Research Letters~{\bf 10} (2003), 191--203.

\bibitem{T1}
B,~Totaro, \emph{ Cheeger manifolds and the classification of biquotients},  Journal of  Differential Geometry~{\bf  61} (2002), no. 3, 397--451.
\end{thebibliography}

Svjetlana Terzi\'c \\
Faculty of Science, University of Montenegro\\
D\v zord\v za Va\v singtona bb, 81000 Podgorica, Montenegro\\
E-mail: sterzic@ac.me

\end{document}